\documentclass[a4paper]{amsart}
\usepackage[utf8]{inputenc}
\usepackage{amscd,amsthm,amsfonts,amssymb,esint}
\usepackage{amsmath}
\usepackage[colorlinks=true]{hyperref}
\usepackage{fullpage}
\usepackage[all]{xy}
\usepackage{mathdots}

\usepackage{mathtools}
\usepackage{mathrsfs}
\usepackage[T1]{fontenc}
\usepackage{lmodern}

    \newcommand{\BC}{{\mathbb {C}}}

    \newcommand{\BQ}{{\mathbb {Q}}} \newcommand{\BR}{{\mathbb {R}}}

     \newcommand{\BZ}{{\mathbb {Z}}}

    \newcommand{\CA}{{\mathcal {A}}} \newcommand{\CB}{{\mathcal {B}}}
     \renewcommand{\CD}{{\mathcal {D}}}
    \newcommand{\CE}{{\mathcal {E}}} 
    \newcommand{\CG}{{\mathcal {G}}}

     \newcommand{\CN}{{\mathcal {N}}}
    \newcommand{\CO}{{\mathcal {O}}} 
     \newcommand{\CR}{{\mathcal {R}}}
    \newcommand{\CS}{{\mathcal {S}}}

     \newcommand{\fp}{{\mathfrak{p}}}

    \newcommand{\Ad}{{\mathrm{Ad}}}
    \newcommand{\ad}{{\mathrm{ad}}}

    \newcommand{\coker}{{\mathrm{coker}}}

    \newcommand{\der}{{\mathrm{der}}}

     \newcommand{\GL}{{\mathrm{GL}}}
    
    \newcommand{\Hom}{{\mathrm{Hom}}}

    \renewcommand{\Re}{{\mathrm{Re}}}

    \newcommand{\SL}{{\mathrm{SL}}}

    \newcommand{\tr}{{\mathrm{tr}}}

    \newcommand{\Vol}{{\mathrm{Vol}}}

    \newcommand{\wt}{\widetilde}
    \newcommand{\wh}{\widehat}

    \newcommand{\ov}{\overline}

    \newcommand{\sk}{\medskip}
    \newcommand{\lra}{\longrightarrow}
    \newcommand{\ra}{\rightarrow} 
    \newcommand{\bs}{\backslash}
    
    \newcommand{\s}{\sk\noindent}

    \theoremstyle{plain}
    \newtheorem{thm}{Theorem}[section] \newtheorem{cor}[thm]{Corollary}
    \newtheorem{lem}[thm]{Lemma}  \newtheorem{prop}[thm]{Proposition}
     \newtheorem{defn}[thm]{Definition}

\theoremstyle{remark} \newtheorem{remark}[thm]{Remark}
\theoremstyle{remark} 
\theoremstyle{remark} 

    \numberwithin{equation}{section}

     \renewcommand{\Ad}{ {\mathrm{Ad}}}
     
     \newcommand{\Kl}{{\mathrm{Kl}}}
     \newcommand{\Pl}{{\mathrm{Pl}}}

\begin{document}

\title{Kloosterman Sums and Pl\"ucker Height Zeta Functions  }

\author{Li Cai} 
\address[L.C.]{Academy for Multidisciplinary Studies, Beijing National Center for Applied Mathematics, Capital Normal University, Beijing, 100048, People's Republic of China
} 
\email{caili@cnu.edu.cn}

\author{Jingsong Chai} 
\address[J.C.]{School of Mathematics, Physics and Finance, Anhui Polytechnic University, Wuhu, Anhui, 241000, People's Republic of China
} 
\email{jingsongchai@hotmail.com}

\author{Yadi Liu} 
\address[Y.L.]{Academy for Multidisciplinary Studies, Beijing National Center for Applied Mathematics, Capital Normal University, Beijing, 100048, People's Republic of China
} 
\email{2220502124@cnu.edu.cn}

\author{Mingyu Yang}
\address[M.Y.]{School of Mathematics, Nanjing University, Jiangsu, 210093, People's Republic of China}
\email{1010850347@qq.com}

\begin{abstract}
We establish nontrivial bounds for Kloosterman sums over $p$-adic fields associated with the longest Weyl element on arbitrary split reductive groups, at arbitrary level. As an application, we obtain the regularity of Bessel distributions (of Whittaker models) for all such groups. 
\end{abstract}



\maketitle
\tableofcontents

\section{Introduction}

%In this paper, we consider the non-trivial bound of Kloosterman sums on split groups. It is a continuation of \cite{CCL26} and \cite{CLYY26}.

%Bounding Kloosterman sums is an important problem with many applications, see for example \cite{AB24, BM24, CCL26, CLYY26, CLPS91, Fri87, Lin26, Man22, Man24, Mia24, Muj24, Sal32, Ste87, Wei48}.

Bounding Kloosterman sums is a fundamental problem in number theory, with many important applications. 

In the classical $\mathrm{GL}_2$ case, Salié \cite{Sal32} obtained early estimates. Weil \cite{Wei48} established the celebrated square-root bound. Such estimates played an important role in Selberg's work on the spectral theory of automorphic forms.

The study of Kloosterman sums beyond $\mathrm{GL}_2$ was pioneered by Bump, Friedberg and Goldfeld \cite{BFG88}. Subsequently, Friedberg \cite{Fri87} and Stevens \cite{Ste87} simultaneously and independently developed the theory for general $\GL_n$. Related developments for orthogonal groups were carried out by Cogdell, Li, Piatetski-Shapiro and Sarnak \cite{CLPS91}. More recently, significant progress has been made for particular families of groups by Assing and Blomer \cite{AB24}, Blomer and Man \cite{BM24}, Man \cite{Man22,Man24}, Linn \cite{Lin26},  Mujdei \cite{Muj24}, Huang (in the appendix of \cite{GSW21}) and  Miao \cite{Mia24}.

A uniform theory of Kloosterman sums for general split groups was initiated by D\k{a}browski \cite{Dab93} and further developed by D\k{a}browski and Reeder \cite{DR98}. In particular, the latter obtained an explicit formula for the generating function of the cardinalities of Kloosterman sets, yielding general trivial bounds. Our recent works \cite{CCL26,CLYY26} further investigate Kloosterman sums on split groups and their connections with Bessel distributions.

In this paper, we establish nontrivial bounds for Kloosterman sums associated with the longest Weyl element on arbitrary split reductive groups at any level.

Our approach to obtain the power-saving counting estimates is through a representation-theoretic analysis of Pl\"ucker height zeta functions, using the Casselman-Shalika formula for spherical Whittaker functions and the Weyl dimension formula. This avoids explicit calculations of Pl\"ucker relations for individual groups.

Let $G$ be a split reductive group over a $p$-adic field $F$ with derived subgroup $G^\der$. Let
$B = AN$ be a Borel subgroup of $G$ with $A$ a maximal torus of $G$ and $N$ the unipotent radical of $B$. Denote by $\delta$ the modulus character of $B$. We choose a representative $\dot{w_0} \in N_G(A)$ of the longest Weyl element. Let $\psi_0$ be the generic character on $N$ given as in Section \ref{sec-notations}.

For an open compact subgroup $K$ of $G$ such that $\psi_0|_{N \cap K} = 1$, we consider the following Kloosterman sum 
\[\Kl(t) =\Kl(t;\psi_0,\dot{w_0},K) = \sum_{x \in X_K(\dot{w_0}t)} \psi_0(u(x))\psi_0(u'(x)), \quad t \in A.\]
Here, the Kloosterman set is
\[X_K(\dot{w_0}t) = (N\cap K)\bs \left(N\dot{w_0}t N \cap K \right)/(N\cap K)\]
together with  two maps 
\[u: X_K(\dot{w_0}t) \ra (N\cap K)\bs N, \quad u': X_K(\dot{w_0}t) \ra N/(N\cap K).\]
In the above, if $[n_1\dot{w_0}tn_2] \in X_K(\dot{w_0}t)$ with $n_1, n_2 \in N$ and $t \in A$, then 
\[u([n_1\dot{w_0}tn_2]) = [n_1], \quad u'([n_1\dot{w_0}tn_2]) = [n_2].\]
	
We say that the Kloosterman sum $\mathrm{Kl}$ has a nontrivial bound (\cite[Definition 3.17]{CCL26}) if there exist constants $\varepsilon>0$ and $C>0$ (possibly depending on $p$) such that
\[\left|\mathrm{Kl}(t)\right|\delta^{1/2-\varepsilon}(t) \leq C, \quad t \in A \cap G^\der.\]

The following is the main result of this paper.

% Add the condition \psi_0|_{N\cap K}=1
\iffalse
\begin{thm}\label{Kloosterman sum-split reductive group} Let $G$ be a split reductive group over $F$. Then for every compact open subgroup $K$, the Kloosterman sum $\Kl(\cdot,\psi_0^{-1},\dot{w_0},K)$ has a nontrivial bound.
\end{thm}
\fi

\begin{thm}\label{Kloosterman sum-split reductive group}
Let $G$ be a split reductive group over $F$.
Then for every compact open subgroup $K\subset G$ such that
$\psi_0|_{N\cap K}=1$, the Kloosterman sum
$\Kl(\cdot,\psi_0^{-1},\dot{w_0},K)$ has a nontrivial bound.
\end{thm}

As an application of Theorem \ref{Kloosterman sum-split reductive group}, we obtain the regularity of Bessel distributions $B_\pi$ on $G$ as follows (See \cite{CCL26} for the survey of this problem). 

Denote by $\CS(G)$ the space of Schwartz functions on $G$. Let
$\pi$ be an irreducible smooth admissible representation on $G$ with $\pi^*$ denoting its linear dual. Let $\pi^\vee$ be the contragredient 
representation of $\pi$, that is, the smooth part of $\pi^*$. The group $G$ then acts on all these spaces and these actions
induce actions of $\CS(G)$. Note that for any $\ell \in (\pi^\vee)^*$ and $f \in \CS(G)$, $\pi^\vee(f)\ell$ is smooth so that 
$\pi^\vee(f)\ell \in (\pi^\vee)^\vee = \pi$. For any $\ell_1 \in \pi^*$ and $\ell_2 \in (\pi^\vee)^*$, the relative character 
associated to $\ell_1$ and $\ell_2$ is defined as
\[ B_{\ell_1,\ell_2}(f) = \ell_1(\pi^\vee(f)\ell_2), \quad f \in \CS(G).\] When $\pi$ is generic with respect to $\psi_0$, and both $\ell_1,\ell_2$ are Whittaker functionals, the associated relative character $B_\pi=B_{\ell_1,\ell_2}$ is called Bessel distribution attached to $\pi$.

\begin{thm}\label{thm-split reductive group}
    Suppose $G$ is a split reductive group over $F$. For any irreducible admissible generic representation $\pi$ of $G$, the Bessel distribution $B_\pi$ with respect to the generic character $\psi_0$ is regular.
    \begin{proof}
        By Theorem \cite[Theorem 3.24]{CCL26}, it is enough to
        prove that for any subset $I$ of the set of simple roots $\Delta$
        for $G$, the Kloosterman sum $\Kl_{M_I}$ on the standard Levi  $M_I$ associated to $G$ with respect to
        the triple $(\psi_0|_{N_{M_I}}^{-1},\dot{w_I},K_I)$
        has a nontrivial bound. Here, $K_I$ is some admissible group of $M_I$ (\cite[Definition 2.25]{CCL26}) and $\{\dot{w_I}\}_{I\subset\Delta}$ is a system of representatives of $\{w_I\}_{I\subset\Delta}$. Applying Theorem \ref{Kloosterman sum-split reductive group} to each $M_I$, we obtain the regularity.
    \end{proof}
\end{thm}

We sketch the proof Theorem \ref{Kloosterman sum-split reductive group}. For this, we assume $G$ is semisimple and simply connected.

Stevens’ approach \cite{Ste87} is a basic method for bounding Kloosterman sums. By considering the action of $A$ on the Kloosterman set, the contribution of each orbit in the Kloosterman sum on $G$ can be decomposed into a product of Kloosterman sums on $\GL_2$. Applying the Weil bound for these Kloosterman sums on $\GL_2$, the problem  is reduced to a counting problem for the Kloosterman set. 

Pl\"ucker coordinates are generators of $F[N \bs G]$. Based on the
homogeneity of Pl\"ucker coordinates, the counting problem is reduced to
that of the Pl\"ucker-constrained sets $S_r$, $r \in \BZ_{\geq 0}^m$ with $m$ the rank of $G$, which are
subsets of affine spaces defined by conditions on Pl\"ucker coordinates
(Definition \ref{defn:plucker-constrained}). This reduction is the main 
result of \cite{CLYY26} (Theorem \ref{main-thm-CLYY}).

For each $r \in \BZ_{\geq 0}^m$ with $R = \sum_i r_i$, we have (Proposition \ref{bound for S_r})
\[|S_r| \leq \sum_{0 \leq \ell \leq R} b_\ell, \quad 
b_\ell = \#\left\{ n \in \CD/N(\CO) \Big| h(n) = \ell \right\}.\]
Here,  $\CD$ consists of elements in $N$ with integral entries for simple roots and $h$ is the Pl\"ucker height (Definition \ref{defn-pl}).

We define the Pl\"ucker height zeta function (Definition \ref{defn-zeta})
\[ Z_\Pl(s) = \sum_{\ell \geq 0} b_\ell q^{-\ell s}.\]
The series is absolutely convergent for $\Re(s)>1$. 

A simple Fourier analysis gives (Lemma \ref{lem-fourier}) 
\[Z_\Pl(s) = \int_{\CO^m} J_s(c)dc, \quad J_s(c) = \int_N f_s(\dot{w_0}n)\ov{\psi_c(n)}dn.\]
Here, $f_s$ is the normalized spherical section of certain unramified principal series on $G$. 

The Casselman-Shalika
formula expresses the spherical Whittaker function $J_s$ on $G$ in terms of 
 Weyl characters of finite-dimensional representations on the Langlands dual group $\wh{G_\ad}$ of the adjoint quotient $G_\ad$
 of $G$ (Theorem \ref{thm-cs}). We obtain that $Z_\Pl(s)$ extends holomorphically to the half-plane (Theorem \ref{holomorphic continuation of F(s)}(1))
 \[\Re(s) > 1 - \frac{1}{M_G}.\]
 Here, $M_G = \max_i a_i$ with $2\rho = \sum_{i=1}^m a_i \alpha_i$ where $2\rho$ is the sum of positive roots and $\alpha_i$, $1 \leq i \leq m$ are the simple roots.
 
 Moreover, the Weyl dimension formula gives the bounding of $Z_\Pl(s)$ for $s$ near its convergence boundary (Theorem \ref{holomorphic continuation of F(s)}(2)).   This estimate for $Z_\Pl(s)$ implies the following bound for  Pl\"ucker constrained sets $S_r$ (Theorem \ref{thm-bound-S_r})
\[ |S_r| \ll q^{\left(1-\frac{1}{2M_G}\right)R}, \quad R = \sum_i r_i.\]
 By Theorem \ref{main-thm-CLYY}, this
 gives a {\em nontrivial bound} for Kloosterman sums on $G$
 with level $K_d$, $d \gg 0$ (Corollary \ref{thm-main}). Here, $\{K_d\}_d$ is the family of
 principal congruence subgroups of $G$. 
 
 By the germ expansion
 of Kloosterman integrals (a generalization of Kloosterman sums)
 in \cite{CCL26}, together with the corresponding bounds for standard Levi subgroups, we obtain non-trivial bounds for Kloosterman sums with arbitrary level. 

It is instructive to compare our approach with that of D\k{a}browski and Reeder \cite[Section 2]{DR98}, whose method yields the {\em trivial} bound for Kloosterman sums.

\s{\bf Acknowledgements.} L. Cai is supported by the National Key R\&D Program of China No.
2023YFA1009702 and National Natural Science Foundation of
China, No. 12371012. J. Chai is supported by National Natural Science Foundation of China, No. 12571010 and No. 12671039.

\s{\bf Use of artificial intelligence.} ChatGPT was used for exploratory computations, editing, searching for references and for verification of results.

\section{Notations and conventions}\label{sec-notations}

Let $F$ be a $p$-adic field. Denote by $\CO$ the ring of integers, $\fp$ the maximal ideal of $\CO$ and $q$  the cardinality of its residue field. Denote by $|\cdot|$ the $p$-adic absolute value on $F$ with $v$ the additive valuation.  We use $\varpi$ to denote a fixed uniformizer of $F$.

Let $G$ be a split (connected) reductive group over  $F$. We shall make no distinction between an algebraic group and its $F$-points.

We fix a maximal torus $A$ of $G$ and 
a Borel subgroup $B = AN$ of $G$. Denote by $Z$ the center of $G$. Let $\Phi^+$ be its positive roots and $\Delta$ its simple roots. Put $L = |\Phi^+|$ and $m=|\Delta|$. Let $\alpha_1,\ldots,\alpha_m\in\Delta$ be the simple roots and $\omega_1,\dots,\omega_m$ be the fundamental weights. 

Denote by
\[2\rho = \sum_{\gamma \in \Phi^+} \gamma.\]
Then
\[\rho = \sum_i \omega_i.\]

Denote by $X^*(A)$ and $X_*(A)$ the groups of characters and
cocharacters on $A$. For any $a \in A$, there is a unique cocharacter $\nu(a) \in X_*(A)$ such that
\[\langle \chi, \nu(a) \rangle = v(\chi(a)), \quad \chi \in X^*(A).\]
Denote by $\delta$ the modulus character of $B$. Then
\[\delta(a) = q^{-\langle 2\rho,\nu(a) \rangle}, \quad a \in A.\]

By the existence theorem for split reductive group schemes (see \cite[Exposé XXV, Theorem 1.1 and Corollary 1.2]{SGA3}), there exists a split reductive group scheme $G_\BZ$ over $\BZ$, a split maximal torus 
$A_\BZ$ and a Borel subgroup $B_\BZ = A_\BZ N_\BZ$
,  together with an isomorphism
\[G_\BZ\otimes_{\BZ} F \stackrel{\sim}{\lra} G\]
such that under this ismorphism, $A_\BZ \otimes_\BZ F$ and 
$B_\BZ \otimes_\BZ F$ to $A$ and $B$, respectively.

We fix a Chevalley system for $G_\BZ$ (\cite[Exposé XXIII, Proposition 6.2]{SGA3})
\[
X_\gamma:\mathbb G_{a,\BZ} \stackrel{\sim}{\lra} N_{\gamma,\BZ}, \quad \gamma\in\Phi.
\]
We shall use the same notation $X_\gamma$ for their base changes to $F$.

Denote by $\CG=G_\BZ \otimes_\BZ \mathcal \CO$.
Then $\CG$ is a split reductive $\CO$-model of $G$. Denote by $\CA$, $\CB$, and $\CN$ the corresponding subgroup schemes obtained by base change from 
$A_\BZ$, $B_\BZ$ and $N_\BZ$ respectively.
In the following, all notations involving $\CO$-point refer to this model.

Let $W=N_G(A)/A$ be the Weyl group. For each simple root $\alpha_i$, let $w_i\in W$ be the corresponding simple reflection. Consider
\[
\dot w_i
=X_{\alpha_i}(1)X_{-\alpha_i}(-1)X_{\alpha_i}(1).
\]
Then 
$\dot w_i\in N_{G_{\BZ}}(A_{\BZ})(\BZ)$ 
represents $w_i$. For each $w\in W$, choose a reduced expression $w=w_{k_1}\cdots w_{k_r}$ and set
\[
\dot w=\dot w_{k_1}\cdots\dot w_{k_r}.
\]
The definition of $\dot w$ is independent on the choice of
the above reduced expression of $w$. In particular, we fix a representative $\dot w_0$ of the longest Weyl element $w_0$ 
which satisfies
\[
\dot w_0\in N_{G_{\BZ}}(A_{\BZ})(\BZ)
\subset N_G(A)(F)\cap G(\CO).
\]
We use this representative throughout the paper.

We fix an ordering $\Phi^+=\{\gamma_1,\dots,\gamma_{L}\}$ by increasing height. With respect to this ordering, every element in $N$ can be written uniquely as 
$$n(\mathbf x)=\prod_{\gamma\in\Phi^+}X_\gamma(x_\gamma),\qquad\mathbf x=(x_\gamma)_{\gamma\in\Phi^+}\in F^{\Phi^+}.$$
If $n = n(\mathbf x)$ for some $\mathbf x \in F^{\Phi^+}$,
then for any $\gamma \in \Phi^+$, we write $x_\gamma(n)$ for $x_\gamma$.

We fix a nontrivial character $\psi_F$ on $F$. We assume that $\psi_F|_{\CO} = 1$, and that $\psi_F|_{\varpi^{-1}\CO}\neq 1$. For each $c = (c_i)_i \in F^m$, denote by
\[\psi_c(n) = \psi_F\left( \sum_i c_ix_{\alpha_i}(n) \right).\]
We shall write $\psi_0 = \psi_c$ if $c_i = 1$ for all $i$.

\section{The Pl\"ucker height}

In this section, we assume $G$ is semisimple and simply connected.  For each $i=1,\dots,m$, let $(\pi_i,V_i)$ be an irreducible (algebraic) representation of $G$ with highest weight $\omega_i \in X^*(A)$.

\begin{lem}\label{lem:int}
    There exists a $G(\CO)$-stable lattice $\Lambda_i$ in $V_i$ with a $\CO$-basis
    $\{e_{ij}\}_j$ of $\Lambda_i$ such that $e_{ij}$ are weight vectors and the
    basis containing a generator $e_i^+$ of the highest-weight space.
\end{lem}
\begin{proof}
    By \cite[Lemma in I.10.4]{Jan03}, there is a $\CG$-stable lattice $\Lambda_i$ in $V_i$ and hence is a $G(\CO)$-stable lattice of $V_i$. Consider the restriction of the action $\CG$ on $\Lambda_i$ to the split torus $\CA$. As the group $\CA$ is diagonalizable, by \cite[(3) in I.2.11]{Jan03}, $\Lambda_i$ decomposes into $A(\CO)$-eigen subspace:
    \[\Lambda_i=\bigoplus_{\mu\in X^*(\CA)}\Lambda_{i,\mu}.\]
    Here,  $X^*(\CA)=\Hom(\CA,\mathbb{G}_{m,\CO})$ and $\Lambda_{i,\mu}$ consists of $v\in\Lambda_i$ such that $\pi_i(a)v=\mu(a)v$ for all $a \in A(\CO)$.

    After extending scalars to $F$, this is the weight-space decomposition of $V_i$
    \[\Lambda_{i,\mu}\otimes_\CO F=V_{i,\mu_F},\quad \Lambda_{i,\mu}=\Lambda_i\cap V_{i,\mu_F}.\]
     Here we use the following canonical isomorphism( \cite[Fact B.1.2]{KP23})
    \[X^*(\CA)\stackrel{\sim}{\lra} X^*(A),\quad \mu\mapsto\mu_F:=\mu\times_\CO F.\]

    The highest-weight space $V_{i,\omega_i}$ is one-dimensional, so its lattice $\Lambda_{i,\omega_i}$ has rank one. Choose an $\CO$-basis in each $\Lambda_{i,\mu}$, choosing a generator $e_i^+$ for the highest-weight summand. The union of these bases is a weight basis of $V_i$ and an $\CO$-basis of $\Lambda_i$.
\end{proof}

For each $i$, denote by $\Lambda_i^* = \Hom_\CO(\Lambda_i,\CO)$ the dual lattic of $\Lambda$. A 
linear functional $\ell_i \in \Lambda_i^*$ is called primitive
if there is no $\ell' \in \Lambda_i^*$ such that
$\ell = \varpi \ell'$.

\begin{lem}\label{lem:int-2}
    For each $i$, there exists a unique primitive linear functional $\ell_i \in \Lambda_i^*$ such that
    \begin{itemize}
        \item $\ell_i \circ\pi_i(n) = \ell_i$ for any $n \in N$.
        \item $\ell_i \circ \pi_i(a) = \dot{w_0}\omega_i(a)\ell_i$
        for any $a \in A$.
        \item $\ell_i(\pi_i(\dot{w_0})e_i^+) = 1$.
    \end{itemize}
\end{lem}
\begin{proof}
    We give a construction of the required $\ell_i$. Consider
    the weight decomposition in Lemma \ref{lem:int}
    \[\Lambda_i = \bigoplus_\mu \Lambda_{i,\mu}.\]
    Denote by $ \lambda_i = w_0\omega_i$. Then $\lambda_i$
    is the lowest weight of $V_i$ and $\Lambda_{i,\lambda_i}
     = \CO e_i^-$ is of rank one. 

     Consider $\ell_i \in \Lambda_i^*$ such that
     \[\ell_i(e_i^-) = 1, \quad \ell_i|_{\Lambda_{i,\mu}} = 0,
     \quad \mu \not= \lambda_i.\]

     For any weight vector $v_\mu \in V_\mu$ and $n \in N$, we have
     \[\pi(n)v_\mu = v_\mu + \sum_{\nu > \mu} v_\nu.\]
     In particular, for any $\mu > \lambda_i$, 
     $\ell_i(\pi(n)v_\mu) = 0 = \ell_i(v_\mu)$ and 
     $\ell_i(\pi_i(n)v_{\lambda_i}) = \ell_i(v_{\lambda_i})$. 
     Therefore, $\ell_i$ is $N$-invariant.

     As the torus $A$ acts on $V_{\lambda_i}$ via 
     $\lambda_i = \dot{w_0}\omega_i$, $\ell_i \pi_i(a) = 
     \lambda_i(a)\ell_i$ for any $a \in A$.

     We have $\pi_i(\dot{w_0})e_i^+ \in V_{\lambda_i}$. As
     we require that $\dot{w_0} \in G(\CO)$, 
     \[\pi_i(\dot{w_0})e_i^+ = ue_i^-, \quad u \in \CO^\times\]
     so that $\ell_i(\pi_i(\dot{w_0})e_i^+)= u$. Modifying
     $\ell_i$ to $u^{-1}\ell_i$, we obtain the require 
     $\ell_i \in \Lambda_i^*$. 

     The uniqueness of such linear functional is clear.
\end{proof}

\begin{defn}[The Pl\"ucker height]\label{defn-pl} For each $i$, we
choose a $G(\CO)$-stable lattice $\Lambda_i$ of $V_i$ and 
an $\CO$-basis $\{e_{ij}\}_j$ of $\Lambda_i$ as in Lemma
\ref{lem:int}. We also choose a primitive linear functional $\ell_i \in \Lambda_i^*$ as in Lemma \ref{lem:int-2}. 
Denote by
\[p_{ij}(n) = \ell_i(\pi_i(\dot{w_0}n)e_{ij}), \quad H_i(n) = \max_j |p_{ij}(n)|, \quad
H(n) = \prod_i H_i(n).\]
\end{defn}

\begin{prop}\label{prop-spherical}
    We have the following.
    \begin{enumerate}
        \item For each $i$, $H_i(n) \geq 1$ and $H_i(nn_0) = H_i(n)$ for any $n \in N$
        and $n_0 \in N(\CO)$. In particular, for each $n \in N$, there exists a unique positive integer $h_i(n) \geq 0$  with $H_i(n) = q^{h_i(n)}$ and we denote by $h(n) = \sum_i h_i(n)$.
       \item Let $G_\ad = G/Z$ be the adjoint quotient associated to $G$. Denote by $A_\ad = A/Z$ and $B_\ad = B/Z$.
       For each $s \in \BC$, consider the induced representation on $G_\ad$
       \[I(s-1) = I_{B_\ad}^{G_\ad} \delta^{(s-1)/2} = \left\{ f:G_\ad \ra \BC \Big| f(bg) = \delta^{s/2}(b)f(g),
       \forall b \in B_\ad, g \in G_\ad\right\}.\]
       Let $f_s \in I(s-1)$ be the unique spherical vector with $f_s(1) = 1$. Then 
       \[f_s(\dot{w_0}n) = H(n)^{-s}, \quad n \in N.\]
    \end{enumerate}
\end{prop}
\begin{remark}
    Here, the reason we prefer $G_\ad$ to $G$ is for the Fourier analysis
    of the zeta function. See Theorem \ref{holomorphic continuation of F(s)} (1).
\end{remark}

\begin{proof}[The proof of Proposition \ref{prop-spherical}]
    (1) The highest-weight vector $e_i^+$ is contained in the integral weight basis of $\Lambda_i$ chosen in Lemma \ref{lem:int}, say $e_{ij_0}=e_i^+$. The unipotent subgroup $N$ of $G$ fixes this vector. Hence for every $n\in N$,
    \[p_{ij_0}(n)=\ell_i(\pi_i(\dot{w_0}n)e_i^+)=\ell_i(\pi_i(\dot{w_0})e_i^+)=1.\]
    Since $H_i(n)=\max_j|p_{ij}(n)|$, this gives $H_i(n)\geq|p_{ij_0}(n)|=1$. In particular, we can write
    $H_i(n) = q^{h_i(n)}$ for some $h_i(n) \geq 0$.
    
    Let $n_0\in N(\CO)$. Because $\Lambda_i$ is $G(\CO)$-stable, $\pi_i(n_0)$ preserves $\Lambda_i$. Let $M$ be the matrix of $\pi_i(n_0)$ with respect to the  $\CO$-basis $\{e_{ij}\}_j$. Then the entries of $M$ are
    contained in $\CO$. For any $n \in N$, we have
    \[(p_{i1}(nn_0),\cdots,p_{id_i}(nn_0)) = (p_{i1}(n),\cdots,p_{id_i}(n)) M\]
    with $d_i = \dim V_i$. This implies that for each $j$
    \[ |p_{ij}(nn_0)| \leq H_i(n)\]
    and
    \[H_i(nn_0) \leq H_i(n).\]
    Similarly, as $\pi_i(n_0^{-1})$ preverses $\Lambda_i$,
    for any $n \in N$, we have
    \[H_i(nn_0^{-1}) \leq H_i(n).\]
    In particular, $H_i$ is right $N(\CO)$-invariant.
    
    (2) For any $n \in N$, consider the Iwasawa decomposition
    \[\dot{w_0}n = uak, \quad u \in N, a \in A, k \in G(\CO).\]
    Then for each $i,j$, 
    \[p_{ij}(n) = (w_0\omega_i)(a)\ell_i(\pi_i(k)e_{ij}).\]
    Similar argument as in the proof of (1), 
    \[\max_j |\ell_i(\pi_i(k)e_{ij})| = \max_j |\ell_i(e_{ij})|.\]
    As $\ell_i$ is primitive, 
    \[\max_j |\ell_i(e_{ij})| = 1.\]
    In particular,
    \[H_i(n) = \max_j |p_{ij}(n)| = |\dot{w_0}\omega_i(a)|, \quad
    H(n) = \prod_i H_i(n) = \prod_i |\dot{w_0}\omega_i(a)|.\]
    
    Therefore,
    \[h(n) = -\sum_i \langle \dot{w_0}\omega_i, \nu(a) \rangle = \langle \rho,\nu(a)\rangle.\]
    We obtain
    \[f_s(\dot{w_0}n) = \delta(a)^{s/2} = q^{-s\langle \rho,\nu(a) \rangle} = H(n)^{-s}.\]
\end{proof}

Consider the Pl\"ucker coordinate polynomials
\[p_{ij}(\mathbf x)=p_{ij}(n(\mathbf x)),\quad\mathbf x\in F^{\Phi^+}.\]

\begin{defn}[Pl\"ucker-constrained set]\label{defn:plucker-constrained}
We fix a section $\CR \subset F$ of $F/\CO$ containing $0$.
We also choose polynomials $q_\gamma \in F[\mathbf x]$ for
each  $\gamma \in \Phi^+$ satisfying
\[q_{\alpha_i} = x_{\alpha_i}, \quad q_\gamma = x_\gamma + f_\gamma, \quad f_\gamma \in F[x_\beta: \beta < \gamma].\]
For $r=(r_1,\dots,r_m)\in\mathbb Z_{\ge0}^m$, define the \emph{Pl\"ucker-constrained set} 
$$S_r=\left\{\mathbf x=(x_\gamma)_\gamma\in F^{\Phi^+}\ \Big|\ x_{\alpha_i}=0,\ v(p_{ij}(\mathbf x))\ge -r_i \text{ for all } i,j, \ q_\gamma(\mathbf x) \in \CR \text{ for all } \gamma \in \Phi^+ \setminus \Delta \right\}.$$
\end{defn}

The main result of \cite{CLYY26} is that the bounding of Pl\"ucker-constrained sets implies the  bounding of Kloosterman sums. See Theorem \ref{main-thm-CLYY}.

Consider the right $N(\CO)$-invariant space
\[\CD = \left\{ n \in N \Big| x_{\alpha_i}(n) \in \CO, \quad 1 \leq i \leq m \right\}.\]

\begin{lem}
    For each positive integer $\ell \geq 0$, there are only
    finitely many $n \in \CD/N(\CO)$ with the Pl\"ucker height
    $h(n) = \ell$. 
\end{lem}
\begin{proof}
We shall give a proof in Lemma \ref{lem-convergence}.  
\end{proof}

For each positive integer $\ell \geq 0$, denote by
\[b_\ell = \#\left\{ n \in \CD/N(\CO) \Big| h(n) = \ell \right\}.\]

\begin{prop}\label{bound for S_r}
   For any $r \in \BZ_{\geq 0}^m$, we have
   \[ |S_r|\leq \sum_{\ell=0}^R b_\ell, \quad R = \sum_i r_i.\]
\end{prop}
\begin{proof}
     We shall prove that
    \[|S_r|\leq\#\left\{ n \in \CD/N(\CO) \Big| h(n)\leq R \right\}.\]
    By \cite[Lemma 3.13]{CLYY26}, there is a bijective map
    \[\left\{\mathbf x=(x_\gamma)_\gamma\in F^{\Phi^+}\ \Big|\ q_\gamma(\mathbf x)\in\mathcal R\text{ for all }\gamma\in\Phi^+\right\} \stackrel{\sim}{\lra} N/N(\CO),\quad \mathbf x\mapsto n(\mathbf x)N(\CO).\]
    In particular, the restriction of the above map to $S_r$
    gives an injection
    \[S_r \xhookrightarrow{\quad} \CD/N(\CO), \quad \mathbf x\mapsto n(\mathbf x)N(\CO).\]

    In the definition of $S_r$, we require that $v(p_{ij}(\mathbf x)) \geq -r_i$ for all $i,j$. We have
    \[H_i(n(\mathbf x)) = \max_j |p_{ij}(\mathbf x)| \leq q^{r_i}\]
    so that
    \[H(n(\mathbf x)) \leq q^R, \quad h(n(\mathbf x)) \leq R.\]
    Therefore, 
    the restriction of the above map $\mathbf x\mapsto n(\mathbf x)N(\CO)$ to $S_r$ gives an injective map
    \[S_r\hookrightarrow\left\{ n \in \CD/N(\CO) \Big| h(n)\leq R \right\}.\]
\end{proof}

\section{The Zeta function}

In this section, we still assume that $G$ is semisimple
and simply connected. We normalize the measure on
$N$ such that $\Vol(N(\CO)) = 1$.

\begin{defn}[The Pl\"ucker height zeta function]\label{defn-zeta} We define the following zeta function for $s \in \BC$
\[Z_\Pl(s) = \int_\CD H(n)^{-s}dn.\]
\end{defn}

\begin{lem}\label{lem-convergence}
    The series defining $Z_\Pl(s)$ is absolutely convergent for $\Re(s) > 1$. In particular, for any $\ell \geq 0$, the cardinality $b_\ell$ is finite and for any $s$ with $\Re(s) > 1$
    \[Z_\Pl(s) = \sum_{\ell \geq 0} b_\ell q^{-\ell s}\]
\end{lem}
\begin{proof}
    By Proposition \ref{prop-spherical}(2), for any $s \in \BC$ and $n \in N$, 
    \[H(n)^{-s} = f_s(\dot{w_0}n)\]
    with $f_s$ the normalized spherical function on the
    induced representation $I(s-1)$.  In particular,
    \[Z_\Pl(s) = \int_\CD f_s(\dot{w_0}n)dn.\]
    The absolutely convergence follows from the absolutely
    convergence of the integral defining the intertwining operator from $I(s-1)$ to $I(1-s)$. This is a special case of \cite[Proposition IV.2.1]{Wal03}.

    As we normalize the measure on $N$ such that 
    $\Vol(N(\CO)) = 1$, for any $s$ with $\Re(s) >1$
    \[Z_\Pl(s) = \sum_{n \in \CD/N(\CO)} H(n)^{-s}
     = \sum_{\ell \geq 0} \sum_{n \in \CD/N(\CO), h(n) = \ell} 
     q^{-\ell s} = \sum_{\ell \geq 0} b_\ell q^{-\ell s}.\]
     In particular, for each $\ell \geq 0$, $b_\ell$ is finite.
\end{proof}

Before we state the main result for the zeta function, we introduce
some notations.

Denote by $\omega_i^\vee \in X_*(A)_\BQ$, $1 \leq i \leq m$ be the fundamental coweights of $G$, that is, $\langle \alpha_i,\omega_j^\vee \rangle = 
\delta_{ij}$. As we assume that $G$ is semisimple and simply connected,
$\omega_i^\vee \in X_*(A_\ad)$. For each $z \in \BZ_{\geq 0}^m$, denote
by
\[\mu_z = \sum_i z_i\omega_i^\vee \in X_*(A_\ad).\]
Denote by $\wh{G_\ad}$ the Langlands dual group of $G_\ad$. In particular,
$\mu_z$ can be viewed as a dominant weight of $\wh{G_\ad}$. Let
$\pi_z$ the irreducible representation on $\wh{G_\ad}$ with the
highest weight $\mu_z$. Denote by
\[P_z(u) = \sum_\chi m(\pi_z,\chi) u^{\langle \rho,\mu_z + \chi \rangle}.\]
Here,  $\chi \in X_*(A_\ad)$ and $m(\pi_z,\chi)$ is the multiplicity
of $\chi$ in $\pi_z$. 

\begin{lem}\label{lem-poly}
    For each $z \in \BZ_{\geq 0}^m$, $P_z(u)$ is
    a polynomial with nonnegative  integral coefficients and
    degree $d_z = \langle 2\rho,\mu_z \rangle$.
\end{lem}
\begin{proof}
    We shall prove that for any weight $\chi \in X_*(A_\ad)$ of $\pi_z$, 
    \[0 \leq \langle \rho,\mu_z + \chi \rangle \leq d_z.\]
    In fact, for any weight $\chi$ of $\mu_z$, we have
    $\mu_z-\chi = \sum_{i=1}^m c_i \alpha_i^\vee$ with
    some $c_i \in \BZ_{\geq 0}$.
Therefore,  $\langle \rho,\mu_z - \chi \rangle \geq 0$ and $\langle \rho,\chi \rangle \leq \langle \rho,\mu_z \rangle$. 
On the other hand, $w_0\chi$ is also a weight of $\mu_z$ so that  \[- \langle \rho,\chi \rangle = \langle \rho,w_0\chi \rangle \leq \langle \rho,\mu_z \rangle\] 

We obtain that
$- \langle \rho,\mu_z \rangle \leq \langle \rho,\chi \rangle \leq \langle \rho,\mu_z \rangle$ and the Lemma holds.
\end{proof}

\begin{thm}\label{holomorphic continuation of F(s)}
   We have the following.
   \begin{enumerate}
       \item 
       In the region $\Re(s) > 1$, we have
       \begin{equation}\label{equ-cs}
       Z_\Pl(s) = (1-q^{-1})^m \CE(q^{1-s})\sum_{z \in \BZ_{\geq 0}^m} q^{-Z}P_z(q^{1-s}).
       \end{equation}
       Here, the polynomial $\CE(u) = \prod_{\gamma \in \Phi^+} (1-q^{-1}u^{\langle \rho,\gamma^\vee\rangle})$ and
     for each $z = (z_i)_i \in \BZ_{\geq 0}^m$, we denote by $Z = \sum_i z_i$. 
       \item 
       Write $2\rho = \sum_i a_i\alpha_i$ with $a_i \geq 1$
       and denote by 
       \[M_G = \max_i a_i.\]
       The right hand side of Equation (\ref{equ-cs}) is absolutely convergent
       for $\Re(s) > 1 - \frac{1}{M_G}$. In particular, this gives the holomorphic 
       continuation of $Z_\Pl(s)$ to $\Re(s) > 1-\frac{1}{M_G}$. Moreover, if 
       $s = 1-\frac{1}{M_G} + \delta$ for a positive real number $\delta$ near to $0$, 
       then 
       \[Z_\Pl(s) \ll \delta^{-(L+m)}.\]
   \end{enumerate}
\end{thm}

We consider the proof of Theorem \ref{holomorphic continuation of F(s)} (1).

For each $c \in \CO^m$, denote by
\[J_s(c) = \int_N f_s(\dot{w_0}n)\ov{\psi_c(n)}dn.\]
Here, $f_s \in I(s-1)$ is the normalized spherical function in  Proposition \ref{prop-spherical}(2).

\begin{lem}[Fourier expansion of $Z_\Pl(s)$]\label{lem-fourier}
    For any $s$ with $\Re(s) > 1$, we have
    \[Z_\Pl(s) = \int_{\CO^m} J_s(c)dc.\]
    Here, the measure on $\CO$ is given with $\Vol(\CO) = 1$.
\end{lem}
\begin{proof}
    By the Fubini theorem, we have
    \[\int_{\CO^m} J_s(c)dc = \int_N f_s(\dot{w_0}n) \int_{\CO^m} 
    \ov{\psi_c(n)}dc dn = \int_N f_s(\dot{w_0}n)\prod_i 1_\CO(x_{\alpha_i}(n))dn = \int_\CD f_s(\dot{w_0}n)dn.\]
    The Lemma now follows from Proposition \ref{prop-spherical}(2). 
\end{proof}

\begin{thm}[Casselman-Shalika] \label{thm-cs}
    Let $c = (c_i)_i \in \CO^m$ with $c_i = u_i\varpi^{z_i}$, $u_i \in \CO^\times$ and $z_i \geq 0$. Denote by $z = (z_i)_i \in \BZ_{\geq 0}^m$. Then
    \[J_s(c) = \CE(q^{1-s}) P_z(q^{1-s}).\]
\end{thm}
\begin{proof}
    This is in fact the Casselman-Shalika formula for spherical
    Whittaker functions. Let $W_s(g)$ be the unique spherical Whittaker function of $I(s-1)$
\[W_s(g) = \int_N f_s(\dot{w_0}ng)\ov{\psi_0(n)}dn.\]
    Then the Casselman-Shalika formula can be stated as following
    (\cite[Theorem 6.5.1]{HKP10})
    \[
    W_s(t_z)=\delta^{1/2}(t_z)\mathcal{E}(q^{1-s})\tr_{\pi_z}(A_s), 
    \quad z \in \BZ_{\geq 0}^m.
    \]
    Here,
    \begin{itemize}
        \item for each $z \in \BZ_{\geq 0}^m$, $t_z = \mu_z(\varpi) \in A_\ad$ and $\tr_{\pi_z}$ is the trace for $\pi_z$.
        \item  the Satake parameter
    $A_s \in \wh{A_\ad}(\BC)$ for $I(s-1)$ is characterized by
    \[\chi(A_s) = q^{-(s-1)\langle \rho,\chi \rangle}, \quad \chi \in X^*(\wh{A_\ad}) = X_*(A_\ad).\]
    \end{itemize}
    We have
    \[\tr_{\pi_z}(A_s) = \sum_\chi m(\pi_z,\chi) q^{(1-s)\langle \rho,\chi \rangle}, \quad \delta^{1/2}(t_z) = q^{-\langle \rho,\mu_z \rangle}.\]
    This gives 
    \[
    W_s(t_z)=\delta^{1-s/2}(t_z)\CE(q^{1-s})P_z(q^{1-s}).
    \]

    Now we fix $c = (c_i)$ with $c_i = u_i\varpi^{z_i}$, $u_i \in \CO^\times$ and $z_i \geq 0$. Consider $z = (z_i)_i \in \BZ_{\geq 0}^m$.
    Then, we have
    \[\alpha_i(t_z) = \varpi^{z_i}, \quad 1 \leq i \leq m.\]
    Moreover, as $X^*(A_\ad)$ is spanned by simple roots $\alpha_i$, $1 \leq i\leq m$, there exists $t_u \in A_\ad(\CO)$ such that
    \[\alpha_i(t_u) = u_i, \quad 1 \leq i \leq m.\]
    Therefore,
    \[\begin{aligned}
        J_s(c) &= \int_N f_s(\dot{w_0}n)\ov{\psi_c(n)}dn \\
        &= \int_N f_s(\dot{w_0}n)\ov{\psi_0\left(\Ad(t_ut_z)n \right)}dn \\
        &= \delta^{-1}(t_z)\int_N f_s\left(\dot{w_0}\Ad(t_ut_z)^{-1}n\right)
        \ov{\psi_0(n)}dn \\
        &= \delta^{s/2-1}(t_z)\int_N f_s(\dot{w_0}nt_z)\ov{\psi_0(n)}dn\\
        &= \delta^{s/2-1}(t_z)W_s(t_z).
    \end{aligned}\]
    Now, the result follows from the above Casselman-Shalika formula
    for $W_s(t_z)$.
\end{proof}

By Lemma \ref{lem-fourier} and Theorem \ref{thm-cs} with $Z = \sum_i z_i$, we have
\[
\begin{aligned}
Z_\Pl(s) &= \sum_{z = (z_i)_i \in \BZ_{\geq 0}^m} q^{-Z}
\int_{(\CO^\times)^m} J_s\left( (u_i\varpi^{z_i})_i \right) \prod_i du_i \\
&= (1-q^{-1})^m \sum_{z= (z_i)_i \in \BZ_{\geq 0}^m} q^{-Z}
J_s\left( (\varpi^{z_i})_i \right) \\
&= (1-q^{-1})^m \CE(q^{1-s})\sum_{z \in \BZ_{\geq 0}^m} q^{-Z}P_z(q^{1-s}).
\end{aligned}
\]
We obtain Theorem \ref{holomorphic continuation of F(s)} (1).

Now, we consider the proof of Theorem  \ref{holomorphic continuation of F(s)}(2). 

\begin{lem}
    Let $z \in \BZ_{\geq 0}^m$ and consider the representation $\pi_z$ on $\wh{G_\ad}$. Then
    \[P_z(1) = \dim \pi_z \leq (1+Z)^L\]
    with $Z = \sum_i z_i$ and $L = |\Phi^+|$.
\end{lem}
\begin{proof}
   By the Weyl dimension formula \cite[Corollary 24.6]{FH91},
we have
\[\dim \pi_z=\prod_{\gamma\in\Phi^+}\frac{\left<\mu_z+\widehat{\rho},\gamma\right>}{\left<\widehat{\rho},\gamma\right>}\]
where $\widehat{\rho} = \sum_i \omega_i^\vee \in X_*(A_\ad)$ and $\gamma \in \Phi^+$ runs over positive
coroots of $\wh{G_\ad}$.

If $\gamma=\sum_ik_i\alpha_i$, then
\[\frac{\left<\mu_z+\widehat{\rho},\gamma\right>}{\left<\widehat{\rho},\gamma\right>} = 1+\frac{\sum_ik_iz_i}{\sum_ik_i}\leq 1+Z.\]
The Lemma now holds. 
\end{proof}

Write $u = q^{1-s}$. Consider
\[1 -\frac{1}{M_G} < \Re(s) \leq 1, \quad 1 \leq |u| < q^{1/M_G}.\]
For each $z \in \BZ_{\geq 0}^m$, By Lemma \ref{lem-poly}, we can write
\[P_z(u)=\sum_{k=0}^{d_z}c_ku^k,\quad c_k\geq 0\]
where
$d_z = \langle 2\rho,\mu_z \rangle = \sum_i a_iz_i$.

As $|u|\geq1$, we have
\[|P_z(u)|\leq\sum_{k=0}^{d_z}c_k|u|^k\leq|u|^{d_z}
\sum_{k}c_k=|u|^{d_z}P_z(1)\leq|u|^{d_z}(1+Z)^L\]
    and
\begin{equation}\label{ineq-P_z}
    q^{-Z}|P_z(u)|\leq q^{-Z}|u|^{d_z}(1+Z)^L=\prod_i(q^{-1}|u|^{a_i})^{z_i} \cdot
    (1+Z)^L.
\end{equation}

Moreover, as $|u|<q^{1/M_G}$, there exists a real number $x<1$ (depending on $|u|$) such that $q^{-1}|u|^{a_i}<x$ for all $i$. Hence we have
\[\sum_{z \in \BZ_{\geq 0}^m} q^{-Z}|P_z(u)|\leq\sum_{z \in \BZ_{\geq 0}^m}(1+Z)^Lx^{Z} < \infty.\]
In particular, the right hand side of Equation 
(\ref{equ-cs}) is absolutely convergent for any $s$
with $\Re(s) > 1 - \frac{1}{M_G}$. Combing with Theorem \ref{holomorphic continuation of F(s)}(1), this gives the holomorphic continuation of $Z_\Pl(s)$ to $\Re(s)>1-\frac{1}{M_G}$.

We estimate $Z_\Pl(s)$ near its convergence boundary. 
Consider
\[s = 1-\frac{1}{M_G}+\delta < 1\]
where $\delta$ is a positive real number and near to
$0$. 
    
By (\ref{equ-cs})
\[Z_\Pl(s) = (1-q^{-1})^m \CE(q^{1-s})\sum_{z \in \BZ_{\geq 0}^m} q^{-Z}P_z(q^{1-s}), \quad \CE(q^{1-s}) = \prod_{\gamma \in \Phi^+} (1-q^{-1}q^{(1-s)\langle \rho,\gamma^\vee\rangle}).\]
As $1-s = \frac{1}{M_G} - \delta < \frac{1}{M_G}$,
$\CE(q^{1-s})$ is bounded by a constant (depending only
on $F$ and $G$). By (\ref{ineq-P_z}), we have
\[Z_\Pl(s)\ll\sum_{z \in \BZ_{\geq 0}^m}(1+Z)^Lq^{-\sum_i[1-(1-s)a_i]z_i}.\]
As $a_i \geq 1$ for all $i$, we have
\[1-(1-s)a_i=1-a_i/M_G+\delta a_i\geq\delta\]
and
\[Z_\Pl(s) \ll\sum_{z \in \BZ_{\geq 0}^m}(1+Z)^Lq^{-\delta Z} = \sum_{Z\geq 0}\binom{Z+m-1}{m-1}(1+Z)^Lq^{-\delta Z}.\]
As
\[\binom{Z+m-1}{m-1} \leq \frac{m^{m-1}}{(m-1)!} (1+Z)^{m-1},\]
we obtain
\[Z_\Pl(s) \ll \sum_{Z \geq 0} (1+Z)^{L+m-1} q^{-\delta Z} \ll (1-q^{-\delta})^{-(L+m)} \ll \delta^{-(L+m)}.\]
In the above, we use
\[\sum_{n=0}^\infty (n+1)^kt^n \ll_k (1-t)^{-(k+1)}\]
and $1-q^{-\delta} \sim (\log q) \delta$ when $\delta \ra 0+$. We complete the proof of Theorem \ref{holomorphic continuation of F(s)} (2).

Applying the bounding of the zeta function $Z_\Pl(s)$ in
Theorem \ref{holomorphic continuation of F(s)} (2), we 
give a bound for Pl\"ucker-constrained sets.

\begin{thm}\label{thm-bound-S_r}
    For any $r \in \BZ_{\geq 0}^m$,
    \[ |S_r| \ll q^{\left(1-\frac{1}{2M_G}\right)R}, \quad R = \sum_i r_i.\]
\end{thm}
\begin{proof}

    By Proposition \ref{bound for S_r}, we have
    \[|S_r|\leq\sum_{\ell=0}^Rb_\ell.\]
    
    Since $b_\ell\geq 0$, Landau's theorem on Dirichlet series with nonnegative coefficients \cite[Theorem 1]{Mau11}, together with
    Lemma \ref{lem-convergence} and
    Theorem \ref{holomorphic continuation of F(s)}, implies that
    \[
    Z_\Pl(s)=\sum_{\ell\geq 0}b_\ell q^{-\ell s}
    \]
    with absolute convergence for $\Re(s)>1-1/M_G$. For any real $s>0$, since each $\ell\leq R$ satisfies $1\leq q^{s(R-\ell)}$, we have
    \[|S_r|\leq q^{sR}\sum_{\ell=0}^Rb_\ell q^{-s\ell}\leq q^{sR}Z_\Pl(s), \quad s > 1 - \frac{1}{M_G}.\]
    By Theorem \ref{holomorphic continuation of F(s)} (2), for $s = 1-\frac{1}{M_G} + \delta$ with the number $\delta$ positive and near to $0$, we have 
    \[Z_\Pl(s) \ll \delta^{-(L+m)}\]
    and
    \[|S_r|\ll q^{(1-\frac{1}{M_G})R}q^{\delta R}\delta^{-(L+m)}.\]
    We may choose $\delta=\frac{1}{2M_G(1+R)}$ (so that $s <1$). Then
    \[|S_r|\ll(1+R)^{L+m}q^{(1-\frac{1}{M_G})R}.\]
    For each $\varepsilon > 0$, $(1+R)^{L+m} \leq C_\varepsilon q^{\varepsilon R}$ for some constant
    $C_\varepsilon$. Consider $\varepsilon = 1/(2M_G)$. Then
    \[|S_r| \ll q^{(1-1/(2M_G))R}.\]
\end{proof}

\section{The proof of Theorem \ref{Kloosterman sum-split reductive group}}

In this section, we firstly still assume that
$G$ is semisimple and simply connected.

Let $\{K_d\}_d$ be the family of principal congruence subgroups of $G$, that is, the kernel of the reduction map $G(\mathcal O)\to G(\mathcal O/\mathfrak p^d)$. Let $\{\Kl_d\}_d$ be the corresponding Kloosterman sums, that is,
\[\Kl_d(t)=\Kl(t;\psi_0,\dot{w_0},K_d),\quad t\in A.\]
We also use $X_d(\dot{w_0}t)$ to denote the Kloosterman set in this case.

For $t\in A$, write $t=t_0\lambda_t(\varpi)$ with $t_0\in A(\mathcal O)$ and $\lambda_t\in X_*(A)$. By D\k{a}browski--Reeder \cite[Theorem 0.3]{DR98}, $X_d(\dot w_0t)$ is nonempty only if $\lambda_t\in\mathbb Z_{\ge0}\Phi^{\vee,+}$. It therefore suffices to consider
$$t=t_0\prod_{i=1}^m\alpha_i^\vee(\varpi^{r_i}),\qquad t_0\in A(\mathcal O),\quad r_i\in\mathbb Z_{\ge0}.$$

\begin{thm}[Theorem 1.3 in \cite{CLYY26}] \label{main-thm-CLYY}
Let $G$ be a split semisimple and simply connected group. Assume there exist constants $\varepsilon>0$ and $C>0$ such that for any $r \in \BZ_{\geq 0}^m$
$$|S_r|\le C\, q^{(1-\varepsilon)R}, \quad R = \sum_i r_i.$$
Then for all sufficiently large $d$, the Kloosterman sum $\mathrm{Kl}_d$ has a nontrivial bound.

\end{thm}

In fact, there is the condition that
$p \nmid \det(C_G)$ with $C_G$ the Cartan matrix
of $G$ in \cite[Theorem 1.3]{CLYY26}. This condition is from the decomposition of the Kloosterman sum on $G$ to products of Kloosterman sums on $\GL_2$ in \cite[Proposition 5.3]{CCL26}. 
We remove this condition in the Appendix (Theorem \ref{thm:general-stevens}).

\begin{cor}\label{thm-main} Let $G$ be a split semisimple and simply connected group. Then for all sufficiently large $d$, the Kloosterman sum $\Kl_d$ has a nontrivial bound
\end{cor}
\begin{proof}
    This is an immediate Corollary from Theorem
    \ref{thm-bound-S_r} and Theorem \ref{main-thm-CLYY}.
\end{proof}

Now, we let $G$ be a general split reductive group over $F$. As before, we fix a Borel subgroup $B =AN$ of $G$ and $\psi = \psi_0$ is the generic character on $N$ given before \cite[Definition 2.1]{CCL26}. Denote by $\{G_i\}_{1 \leq i \leq s}$ the set of minimal nontrivial normal smooth connected subgroups of $G^\der$. By \cite[Theorem in Section 27.5]{H80}, the groups $G_i$ are simple and the multiplication homomorphism
\[G_1 \times \cdots \times G_s \lra G^\der\]
is an isogeny map. These groups $G_i$ are called simple factors of $G$. We have the following isogeny map
\[\wt{G_1}\times \cdots \times\wt{G_s}\lra G^\der\]
where $\wt{G_i},1\leq i\leq s$ is the universal cover of the simple factor $G_i$ of $G^\der$. For each $1\leq i\leq s$, the group $\wt{G_i}$ is simple, simply connected group. 

An open compact subgroup $K\subset G$ with $\psi_0|_{N \cap K} = 1$ is called admissible (with aspect to $\psi_0$) if for any $z \in Z^\der$ and $n \in N$ with $zn \in K$, we have $z = 1$ and $n \in K$, where $Z^\der$ is the center of $G^\der$ (\cite[Definition 2.25]{CCL26}).

\begin{prop}\label{Kloosterman sum for admissible group of split group}
    Let $G$ be a split reductive group. Then there exists an admissible group $K_G$ such that the Kloosterman sum $\Kl(\cdot,\psi_0^{-1},\dot{w_0},K_G)$ has a nontrivial bound.
    \begin{proof}
        By Corollary \ref{thm-main}, for every $1\leq i\leq s$, there exists $\varepsilon_i>0,1\leq i\leq s$ such that the Kloosterman sum $\Kl_{\wt{G_i}}(t_i,\psi_0^{-1}|_{N_i},\dot{w_{0,i}},K_{d_i})$ has a nontrivial bound.
        
        By Lemma \cite[Lemma 3.23 (1)]{CCL26}, the product of $K_{d_i}$ gives an admissible group $K_{G^\der}$ of $G^\der$. By Lemma \cite[Lemma 3.20]{CCL26} and Lemma \cite[Lemma 3.23]{CCL26}, for a representatives $\dot{w_0}$ of the longest Weyl element in $W$, the Kloosterman sums $\Kl_{G^\der}(\cdot;\psi_0^{-1},\dot{w_0},K_{G^\der})$ has a nontrivial bound. Here, we use the fact that the modular characters are compatible under isogeny maps. 
        
        By Lemma \cite[Lemma 3.22]{CCL26}, there exists an admissible group $K_{G}$ of $G$ and a representative $\dot{w_0}$ of the longest Weyl element in $W$ such that the Kloosterman sums $\Kl_{G}(\cdot;\psi_0^{-1},\dot{w_0},K_{G})$ has a nontrivial bound. By \cite[Lemma 3.21 (2)]{CCL26}, the above bound holds for any representative of longest Weyl element $w_0\in W$.
    \end{proof}
\end{prop}
\begin{proof}[The proof of Theorem \ref{Kloosterman sum-split reductive group} ]
    Now let $K\subset G$ be a compact open subgroup such that
    $\psi_0|_{N\cap K}=1$ and $w^0$ is the relevant representative of $w_0$ considered \cite{CCL26}. Consider the Kloosterman integral $I(t,1_{w^0K})$. By \cite[Lemma 3.19]{CCL26}, we have
        \[I(t,1_{w^0K})\stackrel{.}{=}\mathrm{Kl}((w^0)^{-2}t; \psi_0^{-1},w^0,K), \quad t \in A.\]

    Apply \cite[Theorem 2.12]{CCL26}, we have the following germ expansion of Kloosterman integrals  
    \[I(t,1_{w^0K}) =\sum_{J\subset\Delta}(K_\emptyset^{J}*\omega_J)(t),\quad t\in A,\]
    where $\omega_J\in\CS(A_J)$ for every $J\subset\Delta$ and $\{K_I^J\}_{I\subset J}$ is a system of Shalika germs.
    
    By \cite[Proposition 2.29]{CCL26}, for each $J$, we can choose the germ function $K_\emptyset^J$ such that
    \[K_\emptyset^{J}(t)\stackrel{.}{=} \Kl_{M_J}((w_J^0)^{-2}t,\psi_0^{-1}|_{N_{M_J}},w_J^0,K_{M_J}),\quad t \in A_\emptyset^J,\]
    where $K_{M_J}$ is the admissible group in Proposition \ref{Kloosterman sum for admissible group of split group}. In particular, $K_\emptyset^J$ is  supported on $A_\emptyset^J(r_J)$ for some $r_J>0$. 

    By the above germ expansion and \cite[Proposition 3.16]{CCL26},
    there exists constant $C$ such that
    \[\Big|I(t,1_{w^0K})\delta^{\frac{1}{2}-\varepsilon}(t)\Big|\leq C,\quad t\in A\]
    holds provided that for every $J\subset\Delta$, for any $r>0$ small enough, there exist some constants $\varepsilon > 0$ and $C_J>0$, such that
    \[|K_\emptyset^J(t)\delta_{B_{M_J}}^{\frac{1}{2}-\varepsilon}(t)|\leq C_J,\quad t\in A_\emptyset^J(r).\]
    
    By Proposition \ref{Kloosterman sum for admissible group of split group}, such a nontrivial bound holds for $K_\emptyset^J(\cdot)=\Kl_{M_J}((w_J^0)^{-2}\cdot,\psi_0^{-1}|_{N_{M_J}},w_J^0,K_{M_J})$. The proof is done.
\end{proof}

\appendix

\section{Generalized Stevens' trick}
We follow the section 3.2 in \cite{CLYY26}. We still consider
$$t=t_0\lambda_t(\varpi),\qquad t_0\in A(\mathcal O),\quad \lambda_t=\sum_{i=1}^m r_i\alpha_i^\vee,\quad r_i\in\mathbb Z_{\ge0}.$$

Using the fixed ordering and root subgroup parameterizations, write every $n\in N$ uniquely as
$$n(\mathbf x)=\prod_{\gamma\in\Phi^+}X_\gamma(x_\gamma),\qquad \mathbf x=(x_\gamma)_{\gamma\in\Phi^+}\in F^{\Phi^+}.$$
For $n\in N$, write $n_\gamma$ for its $\gamma$-coordinate. Recall that $N_d=\prod_{\gamma\in\Phi^+}X_\gamma(\mathfrak p^d)$ for $d\ge0$. For a positive integer $\ell$, put
$$N_{-\ell}:=\prod_{\gamma\in\Phi^+}X_\gamma\bigl(\mathfrak p^{-\ell\operatorname{ht}(\gamma)}\bigr),$$
where $\operatorname{ht}(\gamma)=\sum_{\alpha\in\Delta}\kappa_\alpha(\gamma)$ for $\gamma=\sum_{\alpha\in\Delta}\kappa_\alpha(\gamma)\alpha$.

Put $A_d=A\cap K_d$, $N_d=N\cap K_d$, and for $d\ge1$ we have
$$A_d=\{t\in A:\chi(t)\in1+\mathfrak p^d\text{ for all }\chi\in X^*(A)\}.$$
Moreover, $\bigcap_{d\ge1}K_d=\{1\}$ and $N_d=\prod_{\gamma\in\Phi^+}X_\gamma(\mathfrak p^d)$ for every $d\ge0$.

Choose $\ell>d$ such that for every $x\in X_d(\dot w_0 t)$ we have $u(x)\in N_d\backslash N_{-\ell}$ and $u'(x)\in N_{-\ell}/N_d$.
For $x=[n_1\dot w_0 t n_2]\in X_d(\dot w_0 t)$, let
$$\theta_i(x)=(n_1)_{\alpha_i}\in\mathfrak p^{-\ell}/\mathfrak p^d,\qquad\theta_i'(x)=(n_2)_{\alpha_i}\in\mathfrak p^{-\ell}/\mathfrak p^d.$$
Write $\theta'=(\theta_i')_{1\le i\le m}:X_d(\dot w_0 t)\to (\mathfrak p^{-\ell}/\mathfrak p^d)^m$ and set $\tau(x)=\min\{v(x),d\}$. 
For $y\in[-\ell,d]^m$, let $X_d(\dot w_0 t)_y$ be the fibre over $y$ under the composition
$$X_d(\dot w_0 t)\xrightarrow{\theta'}(\mathfrak p^{-\ell}/\mathfrak p^d)^m
\xrightarrow{\tau}[-\ell,d]^m.$$

In this section, we prove the following result
\begin{thm}\label{main-thm-CLYY 2}
Let $G$ be a split semisimple and simply connected group. Assume there exist constants $\varepsilon>0$ and $C>0$ such that 
$$|S_r|\le C\, q^{(1-\varepsilon)\sum_{i=1}^m r_i}, \quad r \in \BZ_{\geq 0}^m.$$
Then for all sufficiently large $d$, the Kloosterman sum $\mathrm{Kl}_d$ has a nontrivial bound.
\end{thm}

Following the proof of Theorem \ref{main-thm-CLYY}
in \cite{CLYY26}, we reduce the above theorem
to the following proposition.
\begin{prop}
    For $d>0$ sufficiently large, there exists an integer $\ell>0$, depending on $d$ and $t$, such that
    $$|\mathrm{Kl}_d(t)|\le C_d\sum_{y\in[-\ell,d]^m}|X_d(\dot w_0 t)_y|\,q^{\frac12\sum_{i=1}^m y_i},$$
    where $C_d$ is a constant depend only on $F,G,d$.
\end{prop}

Assume throughout that $d$ is large enough so that the $p$-adic exponential and logarithm give isomorphisms between \(\mathfrak p^r\) and \(1+\mathfrak p^r\) for all \(r\ge d\). Let
\[
f:X_*(A)\to \mathbb Z^m,\qquad f(\lambda)=\bigl(\langle \alpha_1,\lambda\rangle,\dots,\langle \alpha_m,\lambda\rangle\bigr).
\]
The size of the finite abelian group \(\coker(f)\) is the determinant of the Cartan matrix $c_C:=\det(C_G)$ of $G$, and set
\[
E_C:=v(c_C),\qquad D:=d+E_C.
\]
For \(r\ge 1\), put
\[
U_{r,\ell}:=\frac{1+\mathfrak p^r}{1+\mathfrak p^\ell},\qquad 
U_{r,\ell}^{m}:=(U_{r,\ell})^{\times m}.
\]
The image of the simple-root map
\[\phi_d\colon A_d\to(1+\fp^d)^m,\quad t\mapsto(\alpha_1(t),\cdots,\alpha_m(t))\]
is
\[
H_{d,\ell}:=\operatorname{im}\bigl(A_d\to U_{d,\ell}^{m}\bigr).
\]

For each \(x\in X_d(\dot{w_0}t)\), define the constrained Kloosterman sum attached to \(x\) by
\[
S_H(x):=
\sum_{\lambda=(\lambda_1,\dots,\lambda_m)\in H_{d,\ell}}
\psi_F\!\left(
\sum_{i=1}^m
\bigl(\theta_i(x)\lambda_i+\theta_{\sigma(i)}'(x)\lambda_i^{-1}\bigr)
\right),\]
where \(\sigma\) is the permutation of \(\{1,\dots,m\}\)
defined by \(-w_0\alpha_i=\alpha_{\sigma(i)}\).

\begin{prop}\label{prop:avg-general}
We have,
\[
\Kl_d(t)=\frac{1}{|H_{d,\ell}|}\sum_{x\in X_d(\dot{w_0}t)}S_H(x).
\]
\end{prop}

\begin{proof}
For fixed \(x\), the function \(t\mapsto \psi(u(t\cdot x))\psi(u'(t\cdot x))\) is constant on cosets of 
\(A_{d,\ell}:=\ker(A_d\to H_{d,\ell})\). Therefore the orbit integral over \(A_d\) 
can be computed by summing over \(H_{d,\ell}\); the rest of the proof is identical 
to that of \cite[Proposition 5.3]{CCL26}. The details are omitted.
\end{proof}

We now prove that \(H_{d,\ell}\) always contains a product subgroup of depth \(D\).

\begin{prop}\label{prop:deep-subgroup}
If \(\ell>D\), then \(U_{D,\ell}^{m}\subseteq H_{d,\ell}\).
\end{prop}

\begin{proof}
Since \(c_C\mathbb Z^m\subseteq \operatorname{im}(f)\), tensoring with \(\mathcal O_F\) gives
\[
(\mathfrak p^{E_C})^m\subseteq \operatorname{im}\bigl(f\otimes_{\mathbb Z}\mathcal O_F\bigr).
\]
Multiplying by \(\mathfrak p^d\) yields
\[
(\mathfrak p^D)^m\subseteq (f\otimes \text{id})\bigl(X_*(A)\otimes_{\mathbb Z}\mathfrak p^d\bigr).
\]
Composing with the \(p\)-adic exponential map, we have \(U_{D,\ell}^{m}\subseteq H_{d,\ell}\).
\end{proof}

Let \(V:=U_{D,\ell}^{m}\) and choose a set of representatives \(\mathscr R\) for the finite quotient \(H_{d,\ell}/V\). Then
\[
H_{d,\ell}=\bigsqcup_{\eta\in\mathscr R}\eta V.
\]

For integers $0<r<\ell$ and $a,b\in\CO$, define
\[
\Kl(a,b;r,\ell)
:=
\sum_{\lambda\in(1+\mathfrak p^r)/(1+\mathfrak p^\ell)}
\psi_F\!\left(
\frac{a\lambda+b\lambda^{-1}}{\varpi^\ell}
\right).
\]

\begin{prop}\label{prop:product-decomp}
For every \(x\in X_d(\dot{w_0}t)\),
\[
S_H(x)=
\sum_{\eta\in\mathscr R}
\prod_{i=1}^m
\Kl\bigl(\varpi^\ell \theta_i(x)\eta_i,\; \varpi^\ell \theta_{\sigma(i)}'(x)\eta_i^{-1}; D,\ell\bigr).
\]
\end{prop}

\begin{proof}
This follows by substituting \(\lambda_i=\eta_i\mu_i\) with \(\mu_i\in U_{D,\ell}\) into the definition of \(S_H(x)\).
\end{proof}

Combining Propositions~\ref{prop:avg-general} and \ref{prop:product-decomp} yields the generalized Stevens’ trick.

\begin{thm}[Generalized Stevens' trick]\label{thm:general-stevens}
With \(D=d+v(c_C)\) and \(\ell>D\),
\[
\Kl_d(t)=
\frac{1}{|H_{d,\ell}|}
\sum_{x\in X_d(\dot{w_0}t)}
\sum_{\eta\in\mathscr R}
\prod_{i=1}^m
\Kl\bigl(\varpi^\ell \theta_i(x)\eta_i,\; \varpi^\ell \theta_{\sigma(i)}'(x)\eta_i^{-1}; D,\ell\bigr).
\]
\end{thm}

\begin{remark}\label{rmk_arl}
We record the following estimates for the generalized formula.
\begin{enumerate}
\item Since \(H_{d,\ell}\) contains \(U_{D,\ell}^{m}\), we have $|H_{d,\ell}| \ge |U_{D,\ell}^{m}| = q^{m(\ell-D)}$. Hence,
\[
\frac{1}{|H_{d,\ell}|} \le q^{-m(\ell-D)}.
\]
\item The number of coset representatives satisfies
\[
|\mathscr R| = \frac{|H_{d,\ell}|}{|U_{D,\ell}^{m}|} \le q^{m(D-d)} = q^{m\cdot  v(c_C)}.
\]
\item Each component \(\eta_i\) of \(\eta\in\mathscr R\subseteq H_{d,\ell}\) belongs to \(U_{d,\ell}\), hence is invertible.
\end{enumerate}
\end{remark}

Fix an integer $k\geq 0$ such that $q^{-k}\leq |2|^2$. By \cite[Lemmas 3.2--3.4]{JY99}, for all sufficiently large $r$ and $\ell>r$, we have
\[
|\Kl(a,b;r,\ell)|\leq q^{\bigl(\min\{v(b),\ell-2r-k\}+k+\ell\bigr)/2},\qquad a,b\in\mathcal O_F,
\]
where $v(0)=+\infty$. Applying this estimate with $r=D$ and using Remark \ref{rmk_arl}, we obtain
\[
\begin{aligned}
    |\Kl_d(t)|&\le q^{-m(\ell-D)}\sum_{x\in X_d(\dot{w_0}t)}\sum_{\eta\in\mathscr R}\prod_{i=1}^mq^{\bigl(\min\{v(\varpi^\ell \theta_{\sigma(i)}'(x)\eta_i^{-1}),\ell-2D-k\}+k+\ell\bigr)/2}\\
    &\le q^{m(D+k/2)} \sum_{x\in X_d(\dot{w_0}t)}\sum_{\eta\in\mathscr R}\prod_{i=1}^mq^{\tau(\theta_{\sigma(i)}'(x))/2}\\
    &\le q^{m(D+k/2)}q^{m\cdot v(c_C)}\sum_{x\in X_d(\dot{w_0}t)} q^{\sum\limits_{i=1}^m \tau(\theta_{\sigma(i)}'(x))/2 }.\\
\end{aligned}
\]
Here we used $\min\{v(c),-2D-k\}\leq\tau(c)$ for $c\in F$, and \(C:=q^{m(D+v(c_C)+k/2)}\) depends only on \(F,G,d\). 

Hence,
\[
\begin{aligned}
    |\Kl_d(t)|&\le C\sum_{y\in [-\ell,d]^m}|X_d(\dot{w_0}t)_y| q^{\sum\limits_{i=1}^m y_i/2 }.\\
\end{aligned}
\]


\begin{thebibliography}{XX}
%\bibitem[AV16]{AV16} J. Adams, D. Vogan, {\it Contragredient representations and characterizing the local Langlands correspondence.} Amer. J. Math., 138(3):657–682, 2016.
%\bibitem[AGK15]{AGK15} A. Aizenbud, D. Gourevitch and A. Komarsky, {\it Vanishing of certain equivariant distributions on $p$-adic spherical spaces, and nonvanishing of spherical Bessel functions}. International Mathematics Research Notices, (2015), no.18, 8471-8483.
%\bibitem[AGS15]{AGS15} A. Aizenbud, D. Gourevitch and E. Sayag, {\it $z$-finite distributions on $p$-adic groups}. Advances in Mathematics 285 (2015), 1376-1414.
\bibitem[AB24]{AB24} E. Assing and V. Blomer. {\it The density conjecture for principal congruence subgroups}. Duke Math. J. 173(2024), no. 7, 1359-1426. 
%\bibitem[Ba97]{Ba97} E. M. Baruch, {\it On Bessel distributions of $GL(2)$ over a $p$-adic field.} J. Number Theory 67 (1997), 190-202.
%\bibitem[Ba01]{Ba01} E. M. Baruch, {\it On Bessel distributions for quasi-split groups.} Transactions American Math. Society, 353, no.7, (2001), 2601-2614.
%\bibitem[Ba04]{Ba04} E. M. Baruch, {\it Bessel distribution for $GL(3)$ over the $p$-adic field.} Pac. J. Math. 217 (1) (2004), 11-27.
%\bibitem[BM]{BM} E. M. Baruch and Z. Mao. {\it A Whittaker-Plancherel inversion formula for $\SL_2(\BR)$.} Journal of Functional Analysis, 238 (2006), 221-244.
%\bibitem[BB05]{BB05} A. Bj\"orner and F. Brenti. {\it Combinatorics of Coxeter Groups}, Graduate Texts in Mathematics, 231, Springer-Verlag, New York, 2005.
\bibitem[BFG88]{BFG88} D. Bump, S. Friedberg, D. Goldfeld, {\em Poincaré series and Kloosterman sums for $\SL(3,\BZ)$}, Acta Arith. 50 (1988), 31--89.
\bibitem[BM24]{BM24} V. Blomer and S. Man, {\it Bounds for Kloosterman sums on GL(n)}.  Math. Ann. 390 (2024), 1171-1200.
%\bibitem[Bum13]{Bum13} D. Bump, {\it Lie Groups}, volume 225 of Graduate Texts in Mathematics. Springer, New York, second edition, 2013.
\bibitem[CCL26]{CCL26} L. Cai, J. Chai and Y. Liu, \emph{Bessel distributions and Kloosterman sums}, arXiv:2606.29852v2 (2026).
%\bibitem[Chai19]{Chai19} J. Chai, {\it On Bessel functions over $p$-adic fields.} International Mathematics Research Notices, (2019), no.3, 673-699.
%\bibitem[CPS05]{CPS05} J.W. Cogdell, I.I. Piatetski-Shapiro, and F. Shahidi, {\it Partial Bessel functions for quasi-split groups}, Automorphic Representations, L-functions and Applications: Progress and Prospects. Walter de Gruyter, Berlin, 2005, 95–128.
%\bibitem[CST17]{CST17} J. W. Cogdell, F. Shahidi, and T-L. Tsai. {\it Local Langlands correspondence for $\GL_n$ and the exterior and symmetric square $\varepsilon$-factors}. Duke Math. J., 166(11):2053–2132, 2017.
\bibitem[CLYY26]{CLYY26} L. Cai, H. Liu, M. Yang and W. Yang, {\em Stevens' Approach to Kloosterman Sums on Split Groups}, arXiv:2610.24399v2 (2026).
\bibitem[CLPS91]{CLPS91}
J. Cogdell, J.-S. Li, I.Piatetski-Shapiro and P. Sarnak. {\it Poincare series for $SO(n,1)$}. Acta Math. 167(1991), 229-285.
\bibitem[SGA3]{SGA3} M. Demazure, A. Grothendieck, {\it Sch\'emas en groupes I, II, III}, Lecture Notes in Math 151, 152, 153, Springer-Verlag, New York (1970).
\bibitem[Dab93]{Dab93} R. D\k{a}browski, {\em Kloosterman sums for Chevalley groups}, Transactions of the American Mathematical Society, 337 (1993), no. 2, 757--769.
\bibitem[DR98]{DR98} R. D\k{a}browski, M. Reeder, {\it Kloosterman sets in reductive groups}. Journal of Number Theory 73 (1998), no.2, 228-255.
\bibitem[FH91]{FH91} W. Fulton and J. Harris, \emph{Representation theory: a first course}, Graduate Texts in Mathematics, vol. 129, Springer-Verlag, New York, 1991.
\bibitem[Fri87]{Fri87} S. Friedberg, {\em Poincar\'e Series for GL(n): Fourier Expansion, Kloosterman Sums, and Algebreo-
Geometric Estimates.}, Math. Z. 196 (1987), 165-188.
%\bibitem[Del84]{Del84} P. Deligne. {\it Le “centre” de Bernstein.} In Representations of reductive groups over a local field, Travaux en Cours, pages 1–32. Hermann, Paris, 1984. Edited by P. Deligne.
%\bibitem[FLO12]{FLO12} B. Feigon, E. Lapid and O. Offen, On representations distinguished by unitary groups. Publ. Math. Inst. Hautes Etudes Sci. 115 (2012), 185-323.
%\bibitem[GH24]{GH24} J. Getz and H. Hahn. {\it An Introduction to Automorphic Representations, With a view toward trace formulae}, volume 300 of Graduate Texts in Mathematics. Springer, 2024.
\bibitem[GSW21]{GSW21} D. Goldfeld, E. Stade, and M. Woodbury, {\it An orthogonality relation for $\GL(4,\BR)$ (with an appendix by Bingrong Huang)}. arXiv: 1910.13586. Forum of Mathematics, Sigma (2021), 1-83.
%\bibitem[Guo]{Guo} J. Guo. {\it Spherical characters on certain $p$-adic symmetric spaces.} preprint, available at \url{https://archive.mpim-bonn.mpg.de/id/eprint/525/}.
\bibitem[HKP10]{HKP10}
T. Haines, R. Kottwitz and A. Prasad. {\it Iwahori-Hecke algebras.} J. Ramanujan Math.Soc. 25, no.2 (2010), 113-145.
%\bibitem[Ha]{Ha} J. Hakim. {\it Admissible distributions on $p$-adic symmetric spaces.} J. Reine Angew. Math. 455 (1994), 1-19.
%\bibitem[HC70]{HC70} Harish-Chandra, {\it Harmonic analysis on reductive $p$-adic groups.} Lecture Notes in Mathematics, 162, Springer-Verlag, Berlin, 1970. Notes by G. van Dijk.
%\bibitem[HC99]{HC99} Harish-Chandra, {\it Admissible invariant distributions on reductive $p$-adic groups.} University Lecture Series, vol. 16, American Mathematical Society, Providence, RI. 1999, Preface and notes by Stephen DaBacker and Paul J. Sally, Jr.
\bibitem[H80]{H80} J. E. Humphreys, {\it Linear Algebraic Groups}, Graduate Texts in Mathematics, 021, Springer-Verlag, 1980. 
%\bibitem[Jac16]{Jac16} H. Jacquet, {\it Germs for Kloosterman integrals, a review}. Contemp. Math. Vol. 664 (2016), 182-195.
%\bibitem[JY96]{JY96} H. Jacquet, Y. Ye, {\it Distinguished representation and quadratic base change for $\GL(3)$}. Transactions American Math. Society, 348, no.3, (1996), 913-939.
\bibitem[Jan03]{Jan03} J. C. Jantzen {\it Representations of algebraic groups}. Mathematical Surveys and Monographs, Vol. 107. American Mathematical Society, Providence, RI, second edition.
\bibitem[JY99]{JY99} H. Jacquet, Y. Ye, {\it Germs of Kloosterman integrals for $\GL(3)$}. Transactions American Math. Society, 351, no.3, (1999), 1227-1255.
%\bibitem[Kal22]{Kal22} T. Kaletha, {\it On the local Langlands conjectures for disconnected groups}, 2022, arXiv:2210.02519.
\bibitem[KP23]{KP23} T. Kaletha, G. Prasad, {\it Bruhat-Tits Theory–A New Approach}. New Math. Monogr., 44 (Cambridge University Press, Cambridge, 2023), xxx+718 pp.
%\bibitem[LM13]{LM13} E. Lapid, Z. Mao, {\it Stability of certain oscillatory integrals.} International Mathematics Research Notices, (2013), no.3, 525-547.
\bibitem[Lin26]{Lin26} J. Linn, {\it Bounds for Kloosterman Sums for $\GL_n$}, Adv. in Math., vol.499, 2026.
\bibitem[Man22]{Man22} S. Man, {\it Symplectic Kloosterman sums and Poincare series.} Ramanujan J. 57 (2022), 707-753.
\bibitem[Man24]{Man24} S. Man. Fourier coefficients of Sp(4) Eisenstein series. Acta Arith., 213(3):227–271, 2024.
\bibitem[Mau11]{Mau11}
B. N. Maurizi,
{\it Extending Landau's theorem on Dirichlet series
with non-negative coefficients}.
Missouri J. Math. Sci. 23 (2011), no. 2, 105--122.
\bibitem[Mia24]{Mia24} Xinchen Miao. {\it Bessel functions and Kloosterman integrals on $GL(n)$.} Journal of Functional Analysis. 286 (2024), 110267
%\bibitem[Mil17]{Mil17} James S. Milne, {\it Algebraic groups: The theory of group schemes of finite type over a field}, Cambridge Studies in Advanced Mathematics, vol. 170, Cambridge
University Press, Cambridge, 2017.
%\bibitem[QY]{QY} Z. Qi and C. Yang. {\it A Whittaker-Plancherel inversion formula for $\SL_2(\BC)$.} J. Anal. Math. 145, No. 1 (2021), 235-256. 
\bibitem[Muj24]{Muj24}
C. Mujdei. {\it Kloosterman sums on orthogonal groups}. arXiv:2410.20569(2024).
%\bibitem[RR]{RR} C. Rader and S. Rallis. {\it Spherical characters on $p$-adic symmetric spaces.} Amer. Jour. Math., 118 (1996), 91-178.
\bibitem[Sal32]{Sal32}
H. Salie. {\it \"Uber die Kloostermanschen Summen $S(u,v:q)$}. Math. Z. 34(1932), 91-109.
%\bibitem[Spr98]{Spr98} T. A. Springer, {\it Linear algebraic groups}, 2nd ed., Progress in Mathematics, vol. 9, Birkh\"auser Boston, Inc., Boston, MA, 1998, DOI 10.1007/978-0-8176-4840-4. MR1642713
\bibitem[Ste87]{Ste87} G. Stevens, {\it Poincar\'e series on $\GL(r)$ and Kloostermann sums}. Math. Annalen 277 (1987), 25-51.
%\bibitem[Ste68]{Ste68} R. Steinberg, {\it Lectures on Chevalley Groups}. New Haven, CT: Yale University, 1968. Notes prepared by John Faulkner and Robert Wilson.
%\bibitem[SV17]{SV17} Y. Sakellaridis and A. Venkatesh, {\it Periods and Harmonic Analysis on Spherical Varieties}. Astérisque 396, Société Mathématique de France, Paris, 2017.
\bibitem[Wal03]{Wal03}
J.-L. Waldspurger, {\em La formule de Plancherel pour les groupes $p$-adiques (d'après Harish-Chandra)}, J. Inst. Math. Jussieu 2 (2003), 235--333.
\bibitem[Wei48]{Wei48}
A. Weil. {\it On some exponential sums}. Proc.Natl.Acad.Sci. USA 34(1948), 204-207.



\end{thebibliography}
\end{document}